\documentclass[10pt,a4paper,twoside]{amsart}
\usepackage[T1]{fontenc}
\usepackage[utf8]{inputenc}
\usepackage[english]{babel}
\usepackage{amsmath}
\usepackage{amsthm}
\usepackage{amssymb}
\usepackage{amsfonts}
\usepackage{amsxtra}
\usepackage{enumerate}
\usepackage{verbatim}
\usepackage{color}
\usepackage[mathscr]{eucal}
\usepackage{graphicx}

\newtheorem*{qu}{Question}

\newtheorem{theorem}{Theorem}[section]
\newtheorem{definition}[theorem]{Definition}
\newtheorem{lemma}[theorem]{Lemma}
\newtheorem{cor}[theorem]{Corollary}

\newtheorem{prop}[theorem]{Proposition}
\newtheorem*{remark}{Remark}
\newtheorem*{exmp}{Example}

\let\phi=\varphi

\newcommand{\integ}[4]{\int\limits_{#1}^{#2}#3\, d#4}

\newcommand{\abb}[3]{#1\colon #2\rightarrow #3}
\newcommand{\real}[1]{\mathbb{R}^{#1}}
\newcommand{\rem}[1]{}

\begin{document}

\title[contact structure on the space of null geodesics]{The contact structure on the space of null geodesics of causally simple spacetimes}

\author{Jakob Hedicke}
\address{Ruhr-Universit\"at Bochum\\ Fakult\"at f\"ur Mathematik\\ Universit\"atsstra\ss e 150\\ 44801 Bochum, Germany}
\email{Jakob.Hedicke@ruhr-uni-bochum.de}
\thanks{This research is supported by the SFB/TRR 191 ``Symplectic Structures in Geometry, Algebra and Dynamics'', funded by the Deutsche Forschungsgemeinschaft.}

\date{\today}

\begin{abstract}
It is shown that the space of null geodesics of a star-shaped causally simple subset of Minkowski space is contactomorphic to the canonical contact structure in the spherical cotangent bundle of $\real{n}$.
In the $3$-dimensional case we prove a similar result for a large class of causally simple contractible subsets of an arbitrary globally hyperbolic spacetime applying methods from the theory of contact-convex surfaces.
Moreover we prove that under certain assumptions the space of null geodesics of a causally simple spacetime embeds with smooth boundary into the space of null geodesics of a globally hyperbolic spacetime.
The characteristic foliation of this boundary provides an invariant of the conformal class of the causally simple spacetime.
\end{abstract}
\maketitle

\section{Introduction}
Consider a Lorentzian spacetime $(M,g)$.
The associated space of (future pointing) null geodesics $\mathcal{N}_g$ is constructed in \cite{Low, Suhr} as the quotient of the submanifold of future pointing null covectors in $T^{\ast}M$ by the actions induced by the canonical Liouville covector field and the co-geodesic flow.
If this space is a smooth manifold the kernel of the canonical Liouville form on $T^{\ast}M$ projects to a contact structure $\xi_g$ on $\mathcal{N}_g$, see  \cite{Low, Suhr, Baut}.
Although $\mathcal{N}_g$ has a smooth structure for all strongly causal spacetimes (\cite{Low}) in many cases it fails to have the Hausdorff property.
In \cite{Suhr} it was shown that $(\mathcal{N}_g,\xi_g)$ is a smooth contact manifold if $(M,g)$ is causally simple and admits an open conformal embedding into a globally hyperbolic spacetime.

For a globally hyperbolic spacetime with spacelike Cauchy hypersurface $\Sigma$ this contact structure is always contactomorphic to the unit cotangent bundle of $(\Sigma,h|_{T\Sigma})$ with its standard contact structure induced by the Liouville form on $T\Sigma$, as shown in \cite{Low}.
Up to now all known causal examples are of this type.
Based on the result in \cite{Suhr} the question arises if the contact structures in the causally simple case are as well contactomorphic to some spherical cotangent bundle of a smooth manifold with its canonical contact structure.

If the spacetime $(M,g)$ embeds into a globally hyperbolic spacetime $(N,h)$ of the same dimension and $(N,h)$ admits a complete conformal vector field, Theorem \ref{thm2} shows using results from \cite{Chekanov} that under certain assumptions $(N_g,\xi_g)$ is contactomorphic to $(\mathcal{N}_h,\xi_h)$.

Theorem \ref{prop2} implies that in many cases the boundary of $\mathcal{N}_g$ considered as a subset of $\mathcal{N}_h$ is a convex surface in the sense of \cite{Giroux1} if  $N$ is $3$-dimensional.
In this case Theorem \ref{thm3} provides results similar to Theorem \ref{thm2} if $(M,g)$ embeds as a certain contractible subset into $(N,h)$.
The proof uses techniques from the theory of convex surfaces that allow to classify contact structures on open subsets with convex boundary (see \cite{Giroux2}, \cite{Honda}).
In particular it follows that the boundary of $\mathcal{N}_g$ in $\mathcal{N}_h$ is smooth.
This allows to look at the characteristic foliation of $\partial\mathcal{N}_g$ induced by the contact structure $\xi_h$.
Applying Proposition \ref{prop0} we show in Section $4$ how the characteristic foliation on this boundary can be used to distinguish different conformal classes of causally simple Lorentzian metrics on $M$ if $(N,h)$ is the $3$-dimensional Minkowski space.

\section{Main results}

Let $M$ be a smooth manifold and $g$ a Lorentzian metric, i.e. a non-degenerate symmetric $(0,2)$-tensor field of signature $(-,+,\cdots,+)$.
We call $v\in TM\setminus \{0\}$ \textit{timelike} if $g(v,v)<0$, \textit{lightlike} or \textit{null} if $g(v,v)=0$ and $v\neq 0$, \textit{causal} if $v$ is timelike or lightlike and \textit{spacelike} if $g(v,v)>0$ or $v=0$.
A \textit{time-orientation} on $M$ is the choice of a timelike vector field $X$ on $M$.
A causal vector is called \textit{future pointing} with respect to a chosen time orientation if $g(v,X)\leq 0$, else it is \textit{past pointing}.

Following \cite{Sanchez} define a \textit{spacetime} to be a connected and time-oriented Lorentzian manifold.
Given a spacetime $(M,g)$ denote by $\mathcal{N}_g$ its \textit{space of future directed null geodesics} (up to affine re-parametrisation), i.e. the space of inextendible future directed curves $\gamma$ that satisfy the geodesic equation and $g(\gamma',\gamma')=0$.

A natural topology on this space can be obtained as follows (\cite{Low90}, \cite{Suhr}):
Let
$$\mathcal{L}^{\ast}M:=\{\theta\in T^{\ast}M\setminus \{0\}| \theta=g(v,\cdot), v \text{ is future pointing, null}\}$$
be the set of all future pointing null covectors.
Then $\mathcal{N}_g$ can be identified with the quotient of $\mathcal{L}^{\ast}M$ with respect to the geodesic flow and the flow of the canonical Liouville covector field.
As shown in \cite{Low90} the quotient map $\abb{i}{\mathcal{L}^{\ast}M}{\mathcal{N}_g}$ naturally endows $\mathcal{N}_g$ with a topology.

The \textit{sky} of a point $p\in M$ is the set of $[\gamma]\in\mathcal{N}_g$ such that $p$ lies on $\gamma$. 

If $(M,g)$ is strongly causal $\mathcal{N}_g$ inherits a smooth structure (see \cite{Sanchez} for the definitions of the causal hierarchy).
Note that this smooth structure is in general not Hausdorff.
In \cite{Low90-2} Low showed that the Hausdorff property is equivalent to the null pseudo-convexity of $(M,g)$.

\begin{exmp}
\begin{itemize}
\item[i)]In the case when $(M,g)$ is globally hyperbolic with Cauchy hypersurface $\Sigma$ one can easily show that $\mathcal{N}_g\cong ST^{\ast}\Sigma$ (see \cite{Low}).
Thus for globally hyperbolic spacetimes $\mathcal{N}_g$ is always a smooth manifold.

\item[ii)]Consider the Minkowski space $\real{1,n}$, i.e.  $\real{n+1}$ with Minkowski metric $\eta=-dx_0^2+dx_1^2+\cdots +dx_n^2$ for the standard coordinates $(x_0,\cdots,x_n)$.
Removing a point $p\in \real{1,n}$ one obtains a strongly causal spacetime such that its space of null geodesics is not Hausdorff:
Every sequence of null geodesic converging to a null geodesic through $p$ has two limits.
\end{itemize}
\end{exmp}

Example ii) shows that strong causality is not enough to ensure that the space of null geodesics is a smooth manifold, since the Hausdorff condition is violated.

In the case when $(M,g)$ embeds into a globally hyperbolic spacetime of the same dimension the following theorem holds, see \cite{Suhr}:

\begin{theorem}\label{thm1}
Let $(M,g)$ be a causally simple spacetime that embeds with an open conformal embedding into a globally hyperbolic spacetime $(N,h)$.
Then $\mathcal{N}_g$ is Hausdorff.
\end{theorem}

\begin{remark}
In \cite[Theorem 2]{Vatan} it is stated that for strongly causal $(M,g)$ the Hausdorff-property of  $\mathcal{N}_g$ implies that $(M,g)$ is causally simple.
\end{remark}

A \textit{contact manifold}  is a smooth manifold $M^{2n+1}$ with a smooth hyperplane distribution $\xi\subset TM$ that is maximally non-integrable, i.e. locally $\xi$ is the kernel of a $1$-form $\alpha$ such that $\alpha\wedge (d\alpha)^{n}$ is a volume form.
If $\alpha$ is globally defined, it is a \textit{contact form} for $\xi$.
We say that two contact manifolds $(M_1,\xi_1)$ and $(M_2,\xi_2)$ are \textit{contactomorphic}, if there exists a diffeomorphism $\abb{f}{M_1}{M_2}$ such that $df(\xi_1)=\xi_2$.
A \textit{contact vector field} is a vector field such that its flow consits of contactomorphisms.

Given a contact form $\alpha$, the Reeb vector field of $\alpha$ is uniquely defined by
\begin{align*}
\alpha(R_{\alpha})&=1\\
d\alpha(R_{\alpha},\cdot)&=0.
\end{align*}
Moreover, for any smooth function $f$ there exists a unique contact vector field $X_f$ called the \textit{contact Hamiltonian vector field} defined by
\begin{align*}
\alpha(X_f)&=f\\
d\alpha(X_f,\cdot)&=df(R_{\alpha})\alpha-df.
\end{align*}
For a detailed overview on contact geometry see for example \cite{Geiges}.

\begin{exmp}\label{ex1}
Consider a Riemannian manifold $(\Sigma,k)$.

Let $\lambda$ be the canonical Liouville form on $T^{\ast}\Sigma$ defined by 
$$\lambda_{\theta}(v):=\theta(d\pi(v)),$$
where $\abb{\pi}{T^{\ast}\Sigma}{\Sigma}$ denotes the canonical projection.
Denote by 
$$ST^{\ast}\Sigma:=\{\theta\in T^{\ast}\Sigma| k^{\ast}(\theta,\theta)=1\}$$
the unit cotangent bundle.
Then one can show that $\xi_{st}:=\ker(\lambda|_{ST^{\ast}\Sigma})$ defines a contact structure on $ST^{\ast}\Sigma$ (see \cite{Geiges}). 
Moreover, the unit cotangent bundles defined by two different Riemannian metrics on $\Sigma$ are contactomorphic.
\end{exmp}

As mentioned above, the space of null geodesics, provided it is a smooth manifold, naturally carries a contact structure:
Let $(M,g)$ be a spacetime such that $\mathcal{N}_g$ is a smooth manifold.
Although the projection map $\abb{i}{\mathcal{L}^{\ast}}{\mathcal{N}_g}$ does not map the Liouville form to a well-defined $1$-form, the kernel of $\lambda|_{\mathcal{L}^{\ast}M}$ projects to a contact structure $\xi_g$ on $\mathcal{N}_g$(see \cite{Low90}).

As mentioned before, in the globally hyperbolic case $(\mathcal{N}_g,\xi_g)$ is always contactomorphic to $(ST^{\ast}\Sigma,\xi_{st})$, where $\Sigma\subset M$ is an arbitrary Cauchy surface.
The main goal of the following results is to analyse the contact structure in the cases other than globally hyperbolic spacetimes.

Since the image of null geodesics is invariant under conformal maps (see e.g. \cite{Beem}) one can show that the space of null geodesics and its contact structure are invariant under conformal diffeomorphisms, i.e. a conformal diffeomorphism of the underlying Lorentzian manifold induces a contactomorphism of its space of null geodesics.
In view of this property, one can construct the space of light rays for the conformal class of a Lorentzian metric as the space of images of null geodesics.
The space of light rays naturally carries a contact structure and is contactomorphic to the space of null geodesics of every element of the conformal class (see \cite{Baut}).

The conformal invariance of the space of null geodesics relates the following definition with the notion of a contact star-shaped subset (see \cite{Chekanov}):

\begin{definition}
Let $(M,g)$ be a spacetime.
A \textbf{conformal Killing vector field} is a vector field on $M$ whose flow consits of conformal diffeomorphisms.

We call a relatively compact open subset $U\subset M$ \textbf{conformally star-shaped} if there exists a complete conformal Killing vector field $X$ such that the following conditions hold:
\begin{itemize}
\item Every flow line of $X$ intersects $\partial U$ at most once.
\item $\bigcup\limits_{t\geq 0}\phi^{X}_t(U)=M$.
\end{itemize}
\end{definition}

\begin{exmp}\label{excs}
Consider the $(n+1)$-dimensional Minkowski space $\real{1,n}$.
The map $\abb{f_s}{\real{n+1}}{\real{n+1}}$, $(t,x)\mapsto (e^st,e^sx)$ for $s\in\real{}$ defines a complete conformal flow on $\real{1,n}$.
Given any relatively compact star-shaped open set $U\subset\real{n+1}$ containing the origin we have $\bigcup\limits_{s\geq 0}f_s(U)=\real{n+1}$.
Up to translation we can assume that $U$ is centred at the origin.
Then every ray in $\real{n+1}$ starting at $0$ intersects $\partial U$ in a unique point, i.e. $U$ is conformally star-shaped in $\real{1,n}$.
\end{exmp}

For a subset $U$ of a globally hyperbolic spacetime $(N,h)$ denote the metric induced on $U$ by $g_U:=h|_U$ and its space of null geodesics by $(\mathcal{N}_U,\xi_U)$.
Moreover, denote by $\abb{i_U}{U}{N}$ the inclusion map and by $\abb{\iota_U}{\mathcal{N}_U}{\mathcal{N}_h}$ the induced map on the spaces of null geodesics sending a null geodesic in $U$ to its extension in $N$.

\begin{theorem}\label{thm2}
Let $(N,h)$ be a globally hyperbolic spacetime with Cauchy surface $\Sigma$.
Let $U\subset N$ be conformally star-shaped such that the map $\iota_U$ is an embedding.
Then 
$$(\mathcal{N}_U,\xi_U)\cong (\mathcal{N}_h,\xi_h)\cong (ST^{\ast}\Sigma,\xi_{st}).$$
\end{theorem}

\begin{exmp}
There exist causally simple subsets $M$ of a globally hyperbolic spacetime $(N,h)$ such that $\iota_M$ is not injective.
Consider $N=\real{}\times S^n$ with the metric $h=-dt^2+g_{st}$, where $g_{st}$ is the standard Riemannian metric on $S^n$ of constant curvature $1$.
Let $M=\real{}\times H$, where $H$ denotes a hemisphere of $S^n$.
Then $(M,h|_M)$ is causally simple because$(H,g_{st}|_H)$ is geodesically convex (see \cite{Suhr}), but all null geodesics in $N$ except those projecting to the equator intersect $M$ infinitely many times.
Note that $\mathcal{N}_M$ is diffeomorphic to $\real{n}\times S^{n-1}$, i.e. to the co-sphere bundle of $\real{n}$.
Given a null geodesic $[\gamma]$ in $\mathcal{N}_h$ every intersection point of $\gamma$ with $\partial M$ uniquely determines an element in $\mathcal{N}_M$.
Thus $\mathcal{N}_M$ is determined by the union over all $t\in \real{}$ of the unit tangent vectors of $T(\{t\}\times \partial H)$ pointing into $\{t\}\times \partial H$.
For $p\in\partial H \cong S^{n-1}$ the set of unit tangent vectors pointing into $H$ is diffeomorphic to $\real{n-1}$, i.e. 
$$\mathcal{N}_M\cong \real{}\times \partial H\times \real{n-1}\cong \real{n}\times S^{n-1}.$$
\end{exmp} 

\begin{cor}
Let $U\subset \real{1,n}$ be a causally simple and star-shaped open subset.
Then 
$$(\mathcal{N}_U,\xi_U)\cong (\mathcal{N}_{\eta},\xi_{\eta})\cong (ST^{\ast}\real{n},\xi_{st}).$$
\end{cor}

\begin{proof}
Due to Example \ref{excs} and Theorem \ref{thm2} the corollary is true for relatively compact star-shaped subsets of $\real{1,n}$.

Let $U$ be an arbitrary star-shaped open subset of $\real{1,n}$.
Then $U$ can be conformally embedded into a causal diamond, i.e. the domain of dependence of a unit circle in $\{0\}\times \real{n}$ (see e.g. \cite{Penrose}).
The image of this embedding is relatively compact and star-shaped in $\real{1,n}$, i.e. its space of null geodesics is contactomorphic to $(\mathcal{N}_{\eta},\xi_{\eta})$.
The conformal embedding of $U$ induces a contactomorphism $(\mathcal{N}_{U},\xi_{U})\cong(\mathcal{N}_{\eta},\xi_{\eta})$
\end{proof}

\begin{definition}\label{def1}
Let $(N,h)$ be globally hyperbolic.
We call an open relatively compact subset $K\subset N$  \textbf{strongly null convex} if the following conditions are satisfied:
\begin{itemize}
\item[1)] The map $\abb{\iota_K}{\mathcal{N}_K}{\mathcal{N}_h}$ is an embedding.
\item[2)]$\partial K$ is the level set of a smooth function $H$ such that $H$ is regular near $\partial K$ and $\mathrm{Hess}_h(H)(v,v)\neq 0$ for all $v\in \mathcal{L}N\cap T\partial K$.
\item[3)] every null geodesic tangent to $\partial K$ intersects $\partial K$ in a unique point.
\end{itemize}
\end{definition}

Given a surface $F$ in a $3$-dimensional contact manifold $(M,\xi)$ its \textit{characteristic foliation} is the singular foliation defined by the singular line bundle $\xi\cap TF$.
This singular line bundle can be defined by a vector field vanishing at the singular points, i.e. for all $p\in F$ such that $T_pF=\xi_p$ (see \cite{Geiges}).

To analise the contact structure for causally simple subsets of a $3$-dimensional globally hyperbolic spacetime,  we will use results from the theory of convex surfaces introduced by Giroux \cite{Giroux1}.

\begin{definition}
A closed surface $F\subset (M,\xi)$ in a $3$-dimensional contact manifold is called \textbf{convex} if there exists a neighbourhood $U$ of $F$ and a contact vector field $Y\in \Gamma (TU)$ transverse to $F$.

The dividing set of $F$ is the multi curve 
$$\Gamma_Y:=\{x\in F|Y_x\in \xi_x\}.$$
\end{definition}

Note that $\Gamma_Y$ is a collection of smooth circles in $F$.
Moreover $\Gamma_Y$ is always transverse to the characteristic foliation (see e.g. \cite{Giroux1}, \cite{Etnyre}).
The dividing set divides the characteristic foliation in the sense of \cite[Definition 4.8.3]{Geiges}.
In particular this property does not depend on the choice of the transverse contact vector field.

Giroux in \cite{Giroux2} and Honda in \cite{Honda} were independently able to classify contact structures on solid tori with convex boundary.
It turns out that the geometry of a solid torus with convex boundary only depends on the dividing set of the boundary and not on its characteristic foliation.

Proposition \ref{prop2} and the results by Giroux, Honda , Kanda and Makar-Limanov (\cite{Giroux2}, \cite{Honda}, \cite{Kanda}, \cite{Makar}) provide a method to compare the contact structure of different causally simple (possibly non conformally star-shaped) subsets in the $3$-dimensional case.

Let $(N,h)$ be a $3$-dimensional globally hyperbolic spacetime.
Consider a strongly null convex subset $K\subset N$.
Denote by $T_K$, $L_K$ and $S_K$ the timelike, lightlike and spacelike part of $\partial K$, i.e. 
\begin{align*}
T_K& :=\lbrace p\in\partial K: h|_{T_p\partial K\times T_p\partial K} \text{ is a Lorentzian metric}\rbrace\\
L_K& :=\lbrace p\in\partial K: h|_{T_p\partial K\times T_p\partial K} \text{ is degenerate} \rbrace\\
S_K& :=\lbrace p\in\partial K: h|_{T_p\partial K\times T_p\partial K} \text{ is a Riemannian metric} \rbrace.
\end{align*}

For a strongly null convex $K$ denote by $\mathcal{L}\partial K$ the set of future pointing null vectors tangent to $\partial K$.
The canonical Liouville vector field on $TM$ induces an $\real{}_{>0}$-action on $\mathcal{L}\partial K$ whose quotient we denote by $\mathcal{L}\partial K/\real{}_{>0}$.
There is a natural bijection $\abb{\sigma}{\partial\mathcal{N}_K}{\mathcal{L}\partial K/\real{}_{>0}}$ mapping $[\gamma]\in\partial\mathcal{N}_K$ to the equivalence class of $\gamma'\cap T\partial K$.

\begin{prop}\label{lem3}
Let $K$ be strongly null convex.
Then $\sigma$ is a diffeomorphism between smooth manifolds.
Moreover the part of $\partial\mathcal{N}_{g_K}$ where the characteristic foliation is not singular defines a two fold cover over $T_K$.
The leaves of the characteristic foliation project to smooth lightlike curves on $\partial K$ and the foliation is singular at a geodesic $[\gamma]$ if and only if $\gamma$ intersects $L_K$.
\end{prop}

\begin{theorem}\label{prop2}
Let $K\subset (N,h)$ be strongly null convex.
Then $\partial \mathcal{N}_{g_K}$ is a convex surface in $\mathcal{N}_h$.
\end{theorem}

\begin{qu}
In all known examples of strongly null convex contractible subsets $K$ of a $3$-dimensional globally hyperbolic spacetime, the timelike and lightlike boundary are similar to the one of the open unit ball in $\real{1,2}$ (see Figure \ref{Abb1}).
Therefore one could ask the question if for arbitrary contractible strongly null convex $K$, $T_K$ is diffeomorphic to $\real{}\times S^1$ and $L_K$ consists of two embedded circles that are the boundary of $T_K$.
\end{qu}

\begin{theorem}\label{thm3}
Let $K$ be a strongly null convex subset of a $3$-dimensional globally hyperbolic spacetime.
Assume that $K$ is contractible through causally simple subsets, $T_K$ is diffeomorphic to $\real{}\times S^1$ and that $L_K$ consists of two embedded circles that are the boundary of $T_K$.
Then
$$(\mathcal{N}_{K},\xi_{K})\cong (ST^{\ast}\real{2},\xi_{st}).$$
\end{theorem}

\begin{remark}
In all known examples of contractible strongly null convex subsets $K$ of a $3$-dimensional globally hyperbolic spacetime, all connected components of the boundary of $\partial\mathcal{N}_{g_K}$ are diffeomorphic to a convex torus with two circles of singularities as described in \cite[Example 2.27]{Etnyre} (see also Figure \ref{Abb1}).
The assumption that $K$ is contractible is necessary, since in general the topology of $\mathcal{N}_{g_K}$ can differ depending on the topology of $K$:

Consider $(\real{n}\setminus\{0\},g)$, where $g$ is a complete Riemannian metric on $\real{n}\setminus\{0\}$ that coincides with the euclidean metric outside of $B^n_{\frac{1}{2}}(0)$, the (euclidean) ball of radius $\frac{1}{2}$.
Let $(N,h):=(\real{}\times \real{n}\setminus\{0\}, -dt^2+g)$ and $K:=B^{n+1}(0)\setminus (\{0\}\times (-1,1)$.
Then $K$ is strongly null convex in $N$ and the timelike/lightlike part of $\partial K$ are identical to the one of $B^{n+1}(0)$ in Minkowski space.

While the space of null geodesics of $\mathcal{N}_{B^{n+1}(0)}\cong \real{n}\times S^{n-1}$ due to Theorem \ref{thm2}, $\mathcal{N}_K$ is diffeomorphic to $(\real{n}\setminus\{0\})\times S^{n-1}$:
The flow defined in Example \ref{excs} defines a conformal flow on $(N,h)$, i.e. $K$ is conformally star-shaped.
\end{remark}

\begin{prop}\label{prop0}
Let $\abb{i_1}{(M,g_1)}{(\real{1,2},\eta)},\abb{i_2}{(M,g_2)}{(\real{1,2},\eta)}$ be conformal embeddings into Minkowski space.
Assume that $i_1(M)$ and $ i_2(M)$ are strongly null convex.
If there exists a conformal diffeomorphism $\abb{F}{(M,g_1)}{(M,g_2)}$ then there is a contactomorphism $\abb{\phi}{(\mathcal{N}_{\eta},\xi_{\eta})}{(\mathcal{N}_{\eta},\xi_{\eta})}$ such that $\phi(\partial\mathcal{N}_{i_1(M)})=\partial\mathcal{N}_{i_2(M)}$, i.e. the characteristic foliations induced by $\xi_{\eta}$ on $\partial\mathcal{N}_{i_1(M)}$ and $\partial\mathcal{N}_{i_2(M)}$ coincide.
\end{prop}

\begin{proof}
The map $i_2\circ F\circ i_1^{-1}|_{i_1(M)}$ defines a conformal diffeomorphism from $i_1(M)$ to $i_2(M)$.
As a conformal map between open subsets of $\real{1,2}$, $F$ can be extended up to a closed set of singularities to all of $\real{1,2}$ (see e.g. \cite{Kuhnel, Schottenloher}):
Every conformal map on a connected open subset of $\real{1,2}$ is a composition of dilatations, translations, elements of $O(1,2)$ and maps of the form
$$x\mapsto \frac{x-\eta(x,x)b}{1-2\eta(x,b)+\eta(x,x)\eta(b,b)},$$
where $b\in \real{3}$ is fixed.
Since $i_1(M)$ and $i_2(M)$ are relatively compact, there exist open neighbourhoods $U,V$ of $i_1(M)$ and $i_2(M)$ such that $\abb{F|_{U}}{U}{V}$ is a conformal diffeomorphism.
Hence $F|_{U}$ lifts to a contactomorphism of open neighbourhoods of $\mathcal{N}_{i_1(M)}$ and $\mathcal{N}_{i_2(M)}$ such that $\partial\mathcal{N}_{i_1(M)}$ is mapped to $\partial\mathcal{N}_{i_2(M)}$.
This contactomorphism can be extended to $\mathcal{N}_{\eta}$.
\end{proof}

In section $4$ we construct a Poincar\'{e} return map that serves as an invariant of the characteristic foliation of $\partial\mathcal{N}_{K}$ for certain strongly null convex $K\subset \real{1,2}$.
In many cases this invariant is easy to compute and can be used together with Proposition \ref{prop0} to distinguish conformal classes of subsets of Minkowski space.

\section{proofs}

\subsection{Proof of Theorem \ref{thm2}} 
Let $(N,h)$ be globally hyperbolic and $U\subset (N,h)$ be conformally star-shaped.
Denote with $(\mathcal{N}_U,\xi_U)$ its space of null geodesics.
Assume $(\mathcal{N}_U,\xi_U)$ is an open subset of $(\mathcal{N}_h,\xi_h)$.

Following \cite{Chekanov} define a relatively compact open subset $V\subset \mathcal{N}_{h}$ to be \textit{contact star-shaped} if there exists a complete contact vector field $Y$ (i.e. its flow $\phi^Y_t$ consists of contactomorphisms) such that
\begin{itemize}
\item Every flow line of $Y$ intersects $\partial V$ at most once.
\item $\bigcup\limits_{t\geq 0}\phi^{Y}_t(V)=\mathcal{N}_h$.
\end{itemize}
Denote $V_t=\phi^Y_t(V)$.

To show that $(\mathcal{N}_U,\xi_U)$ is contactomorphic to $(\mathcal{N}_h,\xi_h)$ we follow the proof in \cite[Proposition 3.1]{Chekanov}:

\begin{lemma}
Let $a<b$ and $c<d$ be real constants.
Then there exists a contactomorphism $\abb{\Phi_{a,b}^{c,d}}{\mathcal{N}_h}{\mathcal{N}_h}$ that coincides with $\phi^Y_{c-a}$ on a neighbourhood of $\partial V_a$ and with $\phi^Y_{d-b}$ on a neighbourhood of $\partial V_b$.
\end{lemma}

\begin{proof}
Assume $a=c$, the general case follows by taking $\Phi_{a,b}^{c,d}=\phi^Y_{c-a}\circ\Phi_{a,d-c+a}^{a,b}$.
Choose a contact form $\alpha$ for $\xi_h$.
Let $a<q_1<q_2<\min\{b,d\}$.
Since $\phi^Y_t$ is complete and $V$ is relatively compact it follows that $\overline{V_{q_1}}\subset V_{q_2}$.
Hence one can find a smooth function $F$ on $\mathcal{N}_h$ such that $F([\gamma])=0$ for $[\gamma]\in V_{q_1}$ and $F([\gamma])=\alpha(Y)$ for $[\gamma]\notin V_{q_2}$.
Then $\Phi_{a,b}^{a,d}$ can be defined as the flow of the contact Hamiltonian vector field of $F$ at time $d-b$.
\end{proof}

\begin{lemma}
Let $V\subset\mathcal{N}_h$ be contact star-shaped.
Then $V$ is contactomorphic to $\mathcal{N}_h$.
\end{lemma}

\begin{proof}
Let $(s_n), (r_n)$ be strictly increasing sequences such that $s_n\rightarrow -1$ and $r_n\rightarrow \infty$.
Due to the previous Lemma a contactomorphism from $V$ to $\mathcal{N}_h$ is given by 
$$\phi([\gamma])=\left\lbrace \begin{array}{ccc}
\phi^Y_{r_1-s_1}([\gamma]) & \text{for } [\gamma]\in V_{s_1} \\
& \\ 
\Phi^{r_n,r_{n+1}}_{s_n,s_{n+1}}([\gamma]) & \text{for } [\gamma]\in V_{s_{n+1}}\setminus V_{s_n}
\end{array}\right.. $$ 
\end{proof}

\begin{lemma}
Let $U\subset (N,h)$ be conformally star-shaped such that $\mathcal{N}_U\subset\mathcal{N}_h$ is open.
Then $\mathcal{N}_U$ is contact star-shaped in $\mathcal{N}_h$.
\end{lemma}

\begin{proof}
Since conformal maps between spacetimes lift to contactomorphisms of their spaces of null geodesics, the conformal Killing vector field $X$ defines a contact vector field $\mathcal{X}$ on $\mathcal{N}_h$ with a complete flow.
Obviously 
$$\bigcup\limits_{t\geq 0}\phi^{\mathcal{X}}_t(U)=\mathcal{N}_h.$$
Assume there exists a $[\gamma]\in\partial\mathcal{N}_U$ and $t_0>0$ such that $\phi^{\mathcal{X}}_{t_0}([\gamma])\in \partial \mathcal{N}_U$.
Then one can find a point $p$ on $\gamma$ such that $\phi^X_{t_0}\in\partial U$.
But this implies that $p\in U$, i.e. $[\gamma]\in \mathcal{N}_U$ contradicting $[\gamma]\in\partial\mathcal{N}_U$.
This shows that every non constant flow line intersects $\partial\mathcal{N}_U$ at most once.
\end{proof}

\subsection{Proof of Proposition \ref{lem3} and Theorem \ref{prop2}}

\begin{lemma}\label{lem1}
Let $ (N,h)$ be a globally hyperbolic spacetime and $K\subset N$ be strongly null convex.
Then $\partial\mathcal{N}_K\subset\mathcal{N}_h$ is smooth. 
\end{lemma}

\begin{proof}
By definition $\mathcal{N}_K$ is an open subset of $\mathcal{N}_h$, i.e it is a smooth manifold.

Let $\abb{H}{N}{[-1,1]}$ be smooth and regular near $\partial K=H^{-1}(0)$.
Assume that $\mathrm{Hess}_h(H)(v,v)\neq 0$ for all $v\in \mathcal{L}N\cap T\partial K$ and that $H(K)>0$ and $H(N\setminus \overline{K})<0$.
Choose $\epsilon>0$ such that $U:=H^{-1}((-\epsilon,\epsilon))$ is a tubular neighbourhood of $\partial K$ and $\mathrm{Hess}_h(H)(v,v)\neq 0$ for all $v\in\mathcal{L}U$ with $dH(v)=0$.

Claim: The map
\begin{align*}
&\abb{G}{\mathcal{N}_h}{[-1,1]}\\
&[\gamma]\mapsto \sup\limits_{p\in \gamma}H(p)
\end{align*}
is smooth near $\partial\mathcal{N}_K$.

Denote by $\mathcal{N}_h(U)$ the set of null geodesics in $N$ intersecting $U$.
Note that $\mathcal{N}_U$ is in general not the space of null geodesics of $U$.
This is only the case if the map $\abb{\iota_U}{\mathcal{N}_{h|_U}}{\mathcal{N}_h}$ induced by the inclusion map is an embedding.
Since $K$ is relatively compact, for $[\gamma]\in\mathcal{N}_h(U)$ there exists a $t_0$ such that $H(\gamma(t_0))=G([\gamma])$.
Note that at $\gamma(t_0)$ one has 
$$dH(\gamma'(t_0))=\frac{d}{dt}|_{t=t_0}H(\gamma(t))=0=h(\mathrm{grad}_h(H),\gamma'(t_0)),$$
i.e. $\gamma'(t_0)$ is tangent to $H^{-1}(c)$ for some $c\in(-\epsilon,\epsilon)$.

Let $[\gamma_0]\in \partial\mathcal{N}_K$.
Clearly $\gamma_0$ is tangent to $\partial K$ at a point $\gamma_0(t_0)$.
Moreover due to property $3$) in Definition \ref{def1} $H(\gamma_0(t_0))$ is the unique global maximum of $H$ along $\gamma_0$, i.e. $G([\gamma_0])=H(\gamma_0(t_0))=0$.
Take a globally hyperbolic open neighbourhood $V\subset U$ around $\gamma(t_0)$.
Then $\mathcal{N}_V\cong ST\Sigma$, where $\Sigma$ is a Cauchy hypersurface in $V$.
W.l.o.g assume that $\gamma(t_0)\in \Sigma$ and the orthogonal projection of $\gamma'(t_0)$ to $\Sigma$ lies in $ST\Sigma$.

Choose $\delta>0$ such that $\gamma_v(t)\in U$ for all $v\in ST\Sigma$ and $t\in (-\delta,\delta)$, where $\gamma_v$ denotes the geodesic induced by $v\in ST\Sigma$.
Define a smooth function
\begin{align*}
&\abb{f}{(-\delta,\delta)\times ST\Sigma}{\real{}}\\
&(t,v)\mapsto h(\mathrm{grad}_h(H),\gamma_v'(t)).
\end{align*}

For a non-trivial smooth family $t_s\subset (-\delta,\delta)$ one has
$$\frac{d}{ds}|_{s=0}f(t_s,v)=\frac{d}{ds}|_{s=0}(t_s)\mathrm{Hess}_h(H)(\gamma_v'(t_0),\gamma_v'(t_0))\neq 0.$$
Hence $df_{(t,v)}\neq 0$ for all $(t,v)\in (-\delta,\delta)\times ST\Sigma$.
Therefore $f^{-1}(0)$ is a smooth co-dimension $1$ submanifold of $(-\delta,\delta)\times ST\Sigma$ containing $\gamma_0'(t_0)$.
Note that property $2$) in the definition of strongly null convex ensures that $f^{-1}(0)$ is transverse to the $(-\delta,\delta)$ component near $\gamma'(t_0)$.

Let $[\gamma_s]$ be a smooth family of null geodesics with $[\gamma_0]=[\gamma]$.
For $s$ close to $0$, there exists a smooth function $\abb{\tau}{\real{}}{\real{}}$ such that $\gamma_s'(\tau(s))$ intersects $f^{-1}(0)$, i.e. $H(\gamma_s(\tau(s)))$ is a local maximum.
Suppose there exists $t_s'\neq \tau(s)$ such that $H(\gamma_s(t'_s))$ is the global maximum of $H\circ\gamma_s$ for all $s\neq 0$.
Then $H(\gamma_s(t'_s))$ converges to $0$ for $s\rightarrow 0$ since $G([\gamma_0])=0$.
Property $3$) in Definition \ref{def1} implies $t'_s\rightarrow t_0$.
This contradicts the transversality of $f^{-1}(0)$ to the $(-\delta,\delta)$ component near $\gamma'(t_0)$.
Therefore for all $s$ close to $0$ $H(\gamma_s(\tau(s)))$ is a global maximum.

It follows that $G$ is smooth near $[\gamma_0]$.

Clearly $\partial \mathcal{N}_K =G^{-1}(0)$.
Moreover $dG\neq 0$ in a neighbourhood around $\partial \mathcal{N}_K$:
Let $[\gamma]\in\partial  \mathcal{N}_K$ with $\gamma(t_0)\in \partial K$.
To show that $dG\neq 0$ near $\partial K$, consider a family of null geodesics $\gamma_s$ with $\gamma_0=\gamma$ such that $G([\gamma_s])=H(\gamma(t_s))$.
Then
\begin{align*}
dG\left(\frac{d}{ds}\mid_{s=0} [\gamma_{s}]\right)&=\frac{d}{ds}\mid_{s=0} H(\gamma_s(t_s))=dH(J(t_0))\frac{d}{ds}\mid_{s=0}(t_s),
\end{align*}
where $J(t)$ denotes the Jacobi field along $\gamma_0$ generated by $\gamma_s$.
Choosing a family of null geodesics such that $J(t_0)$ is transverse to $\partial K$, it follows that $dG\neq 0$ near $\partial K$.
Using the implicit function theorem, this shows that $\partial  \mathcal{N}_K$ is smooth.
\end{proof}

\begin{cor}\label{prop1}
Let $(N,h)$ be globally hyperbolic.
Let $\abb{i}{M\times[0,1]}{N}$ be an isotopy of embeddings such that $\partial i_s(M)$ is strongly null convex.
Assume there exists an isotopy $\abb{H}{[0,1]\times N}{\real{}}$ such that $\partial i_s(M)=H_s^{-1}(0)$, $H_s$ is regular near $\partial i_s(M)$ and  $\mathrm{Hess}_h(H_s)(v,v)\neq 0$ for all $v\in \mathcal{L}\partial i_s(M)$.
Then $\mathcal{N}_{i_0^{\ast}h}$ is diffeomorphic to $\mathcal{N}_{i_1^{\ast}h}$.
\end{cor}

\begin{proof}
Since $[0,1]$ is compact, it suffices to show that $\mathcal{N}_{i_0^{\ast}h}$ is diffeomorphic to $\mathcal{N}_{i_{\epsilon^{\ast}h}}$ for some $\epsilon>0$.
Let $U$ be a neighbourhood of $\partial\mathcal{N}_{i_0(M)}$.
Choose $\epsilon>0$ such that $\partial\mathcal{N}_{i_{s}(M)}\subset U$ for all $s\in [0,\epsilon]$.
Recall that $\mathcal{N}_h(U)$ denotes the set of null geodesics intersecting $U$.
Assume that for every $s\in[0,\epsilon]$,  $H_s(i_s(M)\cap U)>0$ and $H(U\setminus \overline{i_s(M)})<0$.
Define the map
\begin{align*}
&\abb{G}{[0,1]\times \mathcal{N}_h(U)}{\real{}}\\
&(s,[\gamma])\mapsto \sup\limits_{p\in\gamma}H_s(p).
\end{align*}

The proof of Lemma \ref{lem1} implies that $G_s$ is smooth for every $s\in[0,\epsilon]$.
Moreover $G$ is smooth in $s$ since $H$ is smooth in $s$.

Let $g$ be a Riemannian metric on $N$.
An isotopy between $\partial i_0(M)$ and $\partial i_{\epsilon}(M)$ is given by $\Phi_s|_{\partial i_0(M)}$, where $\abb{\Phi_s}{U}{U}$ denotes the projection of the time-$s$ map of the time dependent vector field $\mathrm{grad}_g(H_s)$ to $U$.
Since $\partial i_s(M)$ is compact \cite[Chapter 8,Theorem 1.3]{Hirsch} implies that $\Phi_{\epsilon}|_{\partial i_0(M)}$ extends to a diffeotopy with compact support on $N$.
This diffeotopy maps $i_0(M)$ diffeomorphically onto $i_{\epsilon}(M)$.
\end{proof}

Next we want to analyse the characteristic foliation of the boundary of $\mathcal{N}_{K}$ for some strongly null convex $K\subset N$.

First note that for any spacetime $(M,g)$ the contact structure in $\mathcal{N}_g$ has an interpretation in terms of Jacobi fields as explained in \cite{Low} and \cite{Baut}:
Let  $[J]\in T\mathcal{N}_g$ and $(-\epsilon,\epsilon)\ni s\mapsto[\gamma_s]$ be a curve in $\mathcal{N}_g$ with $\frac{d}{ds}|_{s=0}[\gamma_s]=[J]$.
For every smooth choice of representatives $\gamma_s$ defines a variation of null geodesics.
Hence $\frac{d}{ds}|_{s=0}\gamma_s$ defines a Jacobi field along $\gamma_0$.
Changing the representatives of $[\gamma_s]$ this Jacobi field changes by a term parallel to $\gamma_0'$.
Thus $[J]$ can be identified with an equivalence class of Jacobi fields along $\gamma_0$ defined by a variation of null geodesics and differing by a term parallel to $\gamma_0'$.
Furthermore since $[J]$ is defined by geodesic variations containing only null geodesics one has $g(\frac{\nabla}{dt}J,\gamma_0')=0$ which implies that $g(J,\gamma_0')$ is constant along $\gamma_0$.
The contact structure on $\mathcal{N}_g$ can be written as
$$\xi_g([\gamma])=\{[J]\in T_{[\gamma]}\mathcal{N}_g|g(J,\gamma')=0\}.$$
Note that due to the observations above this expression is well defined and does not depend on the choice of representatives.

Since $K$ is strongly null convex, $\partial\mathcal{N}_{g_K}$ is determined by the null vectors tangent to $\partial K$, i.e. $[\gamma]\in \partial\mathcal{N}_{g_K}$ if and only if $\gamma'$ is tangent to $\partial K$ at a unique point.

\begin{lemma}\label{lem2}
The characteristic foliation of $\partial\mathcal{N}_{g_K}$ is determined by the equivalence classes of  Jacobi fields $[J]$ with initial conditions
\begin{align*}
J(0)&\in ((\gamma'(0))^{\perp}\cap T\partial K)\\
\frac{\nabla}{dt}J(0)&\in (\gamma'(0))^{\perp},
\end{align*}
where $J\in [J]\in T_{[\gamma]}\partial\mathcal{N}_K$ and $(\gamma'(0))^{\perp}:=\{v\in T_{\gamma(0)}N|h(v,\gamma'(0))=0\}$.
\end{lemma}

\begin{proof}
Let $\gamma$ be a null geodesic with $\gamma'(0)\in T\partial K$.
As described above the tangent space $T_{[\gamma]}\partial\mathcal{N}_{g_K}$ can be identified with equivalence classes of Jacobi fields along $\gamma$ that arise from variations of null geodesics tangent to $\partial K$.
Thus an equivalence class in this tangent space is uniquely determined by a Jacobi field  $J$ with the initial conditions $J(0)\in T_{\gamma(0)}\partial K$ and $\frac{\nabla}{dt}J(0)\in (\gamma'(0))^{\perp}$.
Since $h(J,\gamma')$ is constant along $\gamma$ one has $J\in \xi_{\eta}\cap T_{[\gamma]}\partial\mathcal{N}_{K}$ if and only if $J(0)\in (\gamma'(0))^{\perp}$.
\end{proof}

\begin{remark}
If $K$ is not relatively compact, there can be null geodesics in $\partial\mathcal{N}_K$ that do not intersect $\overline{K}$.
Consider $(N,h)=(\real{}\times \real{}\times S^{1},-dt^2+g_C)=(\real{}\times \real{}\times S^{1},-dt^2+ds^2+d\theta^2)$, where $\theta$ denotes the angle of a point in $S^1\subset \real{2}$ with respect to the vector $(1,0)$.
Let $K=\real{}\times (-1,1) \times S^{1}$.
Then $K$ is causally simple since $(-1,1) \times S^{1}$ is geodesically convex in $(\real{}\times S^{1},g_C)$.
Moreover $\mathcal{N}_{g_K}$ is an open subset of $\mathcal{N}_h$.
The null geodesics in $N$ are up to parametrisation of the form $(s,\gamma(s))$, where $\gamma$ is a geodesic for $(\real{}\times S^{1},g_C)$.
Every point in $(\real{}\times S^{1},g_C)$ lies on two closed geodesics, all other geodesics are complete and intersect $(-1,1) \times S^{1}$.
Thus all null geodesics in $N$ intersect $\overline{K}$ except for the geodesics starting outside of $\overline{K}$ such that $\gamma(s)$ is closed. 
Hence $\overline{\mathcal{N}_{g_K}}=\mathcal{N}_h$.
\end{remark}

\begin{proof}[Proof of Proposition \ref{lem3}]

We first show that the set $\mathcal{L}\partial K$ of future pointing null vectors tangent to $\partial K$ is smooth.
Let $H$ be defined like in the proof of Lemma \ref{lem1}.
Consider the map
\begin{align*}
&\abb{F}{\mathcal{L}N}{\real{2}}\\
&v\mapsto (H(\pi(v)),dH(v)),
\end{align*}
where $\abb{\pi}{\mathcal{L}N}{N}$ denotes the projection map.
Then $\mathcal{L}\partial K=F^{-1}(0)$.
Let $v\in \mathcal{L}\partial K$ and $\gamma_v$ be the geodesic with $\gamma'(0)=v$.
Then 
$$\frac{d}{dt}|_{t=0}(H(\pi(\gamma_v'(t)),dH(\gamma_v'(t)))=(0,\mathrm{Hess}_h(H)(v,v))\neq 0.$$
If $\gamma_s$ is a variation of null geodesics with $\gamma_s'(0)=v$ such that $\frac{d}{ds}|_{s=0}\gamma_s(0)$ is transverse to $\partial K$ one has
$$\frac{d}{ds}|_{s=0}H(\pi(\gamma_s'(0)))=dH(\frac{d}{ds}|_{s=0}\gamma_s(0))\neq 0.$$
This implies that $dF$ is surjective near $\mathcal{L}\partial K$, i.e. $\mathcal{L}\partial K$ and therefore $\mathcal{L}\partial K/\real{}_{>0}$ are smooth.
Let $[\gamma]\in \partial\mathcal{N}_K$ with $\gamma(0)\in \partial K$.
Take a globally hyperbolic neighbourhood  $V$ around $\gamma(0)$ with Cauchy hypersurface $\Sigma$ such that  $\gamma(t_0)\in \Sigma$ and the part of $\gamma'(t_0)$ tangent to $\Sigma$ lies in $ST\Sigma$.
Then $\mathcal{N}_V\cong ST\Sigma$ is an open neighbourhood of $[\gamma]$.
The map $\sigma$ is the restriction to $\partial\mathcal{N}_K\cap\mathcal{N}_V$ of the diffeomorphism from $\mathcal{N}_V$ to $ST\Sigma$.
Hence the maps $\sigma$ is a diffeomorphism and $\sigma\circ \pi$ is smooth.

Recall that the boundary of $K$ can be divided into the timelike part $T_K$, consisting of the points where $h|_{\partial K}$ is a Lorentzian metric, the spacelike part $S_K$ where $h|_{\partial K}$ is Riemannian and the remaining null part $L_K$.
Since $\partial K$ is $2$-dimensional, for every $p\in T_K$ there are two independent lightlike directions in $T_P\partial K$.
Moreover there exists one lightlike direction if $p\in L_K$ and no lightlike direction if $p\in S_K$.
Let $[J]\in T_{[\gamma]}\partial \mathcal{N}_{g_K}$ for some $[\gamma]$ with $\gamma(0)\in\partial K$.
Due to Lemma \ref{lem2} $[J]\in \xi_{[\gamma]}$ if and only if $J(0)\in (\gamma'(0))^{\perp}\cap T\partial K$.
Hence $[\gamma]$ is a singular point of the characteristic foliation  if and only if $\gamma(0)\in L_K$ since then $T_{\gamma(0)}\partial K=(\gamma'(0))^{\perp}$.
If $Y\in \Gamma(T\partial\mathcal{N}_{g_K})$ defines the characteristic foliation and $[\gamma]$ is not singular, $Y_{[\gamma]}$ projects to the tangent vector of a lightlike curve through $\gamma(0)$ on $\partial K$.
\end{proof}

\begin{proof}[Proof of Theorem \ref{prop2}]
Since $N$ is in particular stably causal, one can choose a smooth time function $\abb{\tau}{N}{\real{}}$ with spacelike level-sets such that $\tau(\overline{K})>0$, i.e. $\tau$ is a function with future pointing timelike gradient (see \cite{Sanchez}).
Due to Proposition \ref{lem3} $\tau|_{\partial K}$ can be lifted to a smooth function $\abb{T}{\partial \mathcal{N}_{K}}{\real{}}$ using the map $\abb{\sigma}{\mathcal{N}_K}{\mathcal{L}\partial K}$.
Choose a contact form $\alpha$ for $\xi_h$ and denote by $\beta$ the restriction of $\alpha$ to $T\partial\mathcal{N}_{g_K}$.

Claim: $\tau$ can be chosen such that $Fd\beta+\beta\wedge dF$ is a volume form.
Then \cite[Lemma 2.10]{Etnyre} implies that $\partial\mathcal{N}_{K}$ is convex.

At the singularities of the characteristic foliation one has $\beta=0$ and $d\beta\neq 0$ since $\alpha$ is a contact form.
Hence in a neighbourhood $U$ around the singularities the claim is always true.

Let $V\subset \partial\mathcal{N}_{K}$ be open such that $\partial\mathcal{N}_{K}=V\cup U$ and $V$ does not contain singularities.
Choose $Y_{1},Y_2\in\Gamma(TV)$ so that $Y_1$ spans the characteristic foliation and $Y_2$ is tangent to the level sets of $T$ and $\beta(Y_2)=2$.
Note that $Y_1$ and $Y_2$ are transverse since they project to a lightlike and spacelike vector field.
Assume that $T$ is strictly increasing along the flow of $Y_1$.

Then
\begin{align*}
Td\beta(Y_1,Y_2)+(\beta\wedge dT)(Y_1,Y_2)&=Td\beta(Y_1,Y_2)+dT(Y_1)\\
&=-T\beta([Y_1,Y_2])+dT(Y_1).
\end{align*}
Therefore
\begin{align*}
Td\beta+\beta\wedge dT>0\Leftrightarrow &\frac{dT(Y_1)}{T}=d(\log(T))(Y_1)>\beta([Y_1,Y_2])
\end{align*}
and since $\partial\mathcal{N}_K$ is compact and $Y_1$ vanishes near the singularities
$$\beta([Y_1,Y_2])<c<\infty.$$
The vector field $Y_2$ only depends on the level sets of $\tau$, not on the value.
Thus $\tau$ can be rescaled without changing the level sets such that $T\beta([Y_1,Y_2])< dT(Y_1)$.

Due to Lemma \ref{lem2} the function $T$ is strictly increasing along the leaves of the characteristic foliation.
In particular, there are no closed leaves since $N$ is causal and a closed leaf would project to a smooth closed lightlike curve in $\partial K$.

\end{proof}

\subsection{Proof of Theorem \ref{thm3}}

\begin{figure}
	\centering
	\includegraphics[scale=0.293]{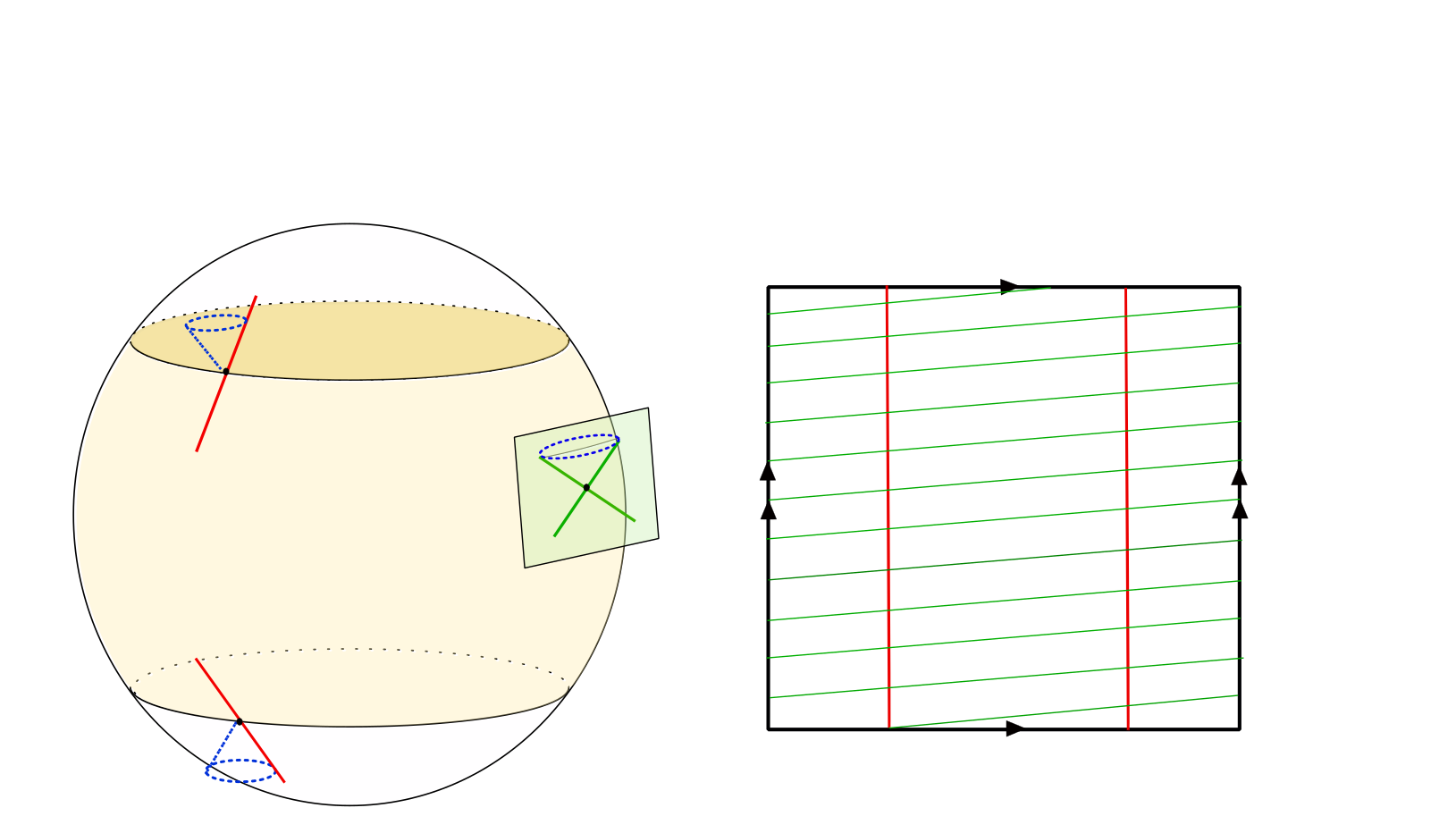}
	\caption{Null geodesics at the boundary of a convex ball $K$ in Minkowski $3$-space and the characteristic foliation of $\partial\mathcal{N}_{K}$.}
	\label{Abb1}
\end{figure}

Claim: $\mathcal{N}_{K}$ is diffeomorphic to $D^2\times S^1$.

Since $K$ is strongly null convex Proposition \ref{lem3} and Theorem \ref{prop2} imply that $\partial \mathcal{N}_{K}$ is a convex torus with two circles of singularities.
The contraction of $K$ through causally simple subsets induces a contraction of $\mathcal{N}_{K}$ to the sky of a point in $K$.
Thus there exists a loop in $\partial \mathcal{N}_{K}$ not null homologous that is contractible in $\mathcal{N}_h$.
Dehn's Lemma (see e.g. \cite[Chapter XVI]{Bing}) implies that there is an embedded disk $D\subset \mathcal{N}_{g_K}$ such that $\partial D\subset \partial \mathcal{N}_{K} $ is homotopic to that loop.
Note that since $\mathcal{N}_{K}$ contracts to a circle, $\pi_2(\mathcal{N}_{K})=0$.
Take a tubular neighbourhood $U\cong D\times (-1,1)$ of $D$ in $\mathcal{N}_{K}$.
Then $\partial (\mathcal{N}_{K}\setminus U)$ and $\partial U$ are homeomorphic to a sphere.
Since $\pi_2(\mathcal{N}_K)=0$ the Poincar\'{e} conjecture proven by Perelman (see \cite{Perelman1, Perelman2, Perelman3}) implies that $U$ and $\mathcal{N}_{K}\setminus U$ are diffeomorphic to a ball.

It follows that $\overline{\mathcal{N}_{K}}$ is obtained by gluing two closed balls at two disjoint discs on their boundary, i.e. $\mathcal{N}_K\cong D^2\times S^1$.

Theorem \ref{prop2} implies that $\partial\mathcal{N}_{K}$ is a convex torus.
The characteristic foliation has two circles of singularities that lie in the homotopy class of the skies, i.e. they are in the homotopy class of a generator of $\pi_1(\partial\mathcal{N}_{K})$.
Given a transverse contact vector field $Y$, the number of circles in $\Gamma_Y$  has to be even and there is no leaf of the characteristic foliation connecting to circles in $\Gamma_Y$ (see \cite{Geiges}).
Moreover a contractible dividing curve would imply that the contact structure on $\mathcal{N}_h$ is overtwisted (see \cite[Theorem 4.8.13]{Geiges}) which contradicts $\mathcal{N}_h$ being contactomorphic to the standard tight contact structure on a spherical co-tangent bundle.
This implies that there exist exactly two dividing circles parallel to the circles of singularities.

 \cite[Theorem 8.2]{Kanda} states that there is a unique contact structure on the solid torus whose boundary is convex and has this characteristic foliation.
 
Let $B^3(0)$ be the open unit ball in $\real{1,2}$.
Then $\partial\mathcal{N}_{B^3}$ is a convex torus divided by the same curves as $\partial\mathcal{N}_{K}$.
Due to Theorem \ref{thm2} $\mathcal{N}_{B^3}$ is contact star-shaped in $\mathcal{N}_{\real{1,2}}$ for a contact vector field $Y$.
By definition $Y$ is transverse to $\partial\mathcal{N}_{B^3}$.
The Giroux flexibility theorem in \cite{Giroux1} allows to isotope $\partial\mathcal{N}_{B^3}$ to a solid torus $T\subset \mathcal{N}_{\real{1,2}}$ such that $\partial T$ has the same characteristic foliation as $\partial \mathcal{N}_{K}$.
Moreover $T$ can be chosen such that $\partial T$ is in an arbitrary small neighbourhood around $\partial\mathcal{N}_{B^3}$.
Then $Y$ is also transverse to $\partial T$ and $T$ is contact star-shaped in $\mathcal{N}_{\real{1,2}}$, i.e.
$$(T,\xi_{\real{1,2}}|_T)\cong ( ST^{\ast}\real{2},\xi_{st}).$$
Together with \cite[Theorem 8.2]{Kanda} this finishes the proof of Theorem \ref{thm3}.

\begin{remark}
The techniques of the proof only work if $(N,h)$ is $3$-dimensional.
One major difficulty is to determine the diffeomorphism type of $\mathcal{N}_K$.
Methods that are recently developed in \cite{Kwon} or \cite{Honda2} could be used to generalise Theorem \ref{thm3} to higher dimensions.
\end{remark}

\section{The characteristic foliation and conformal classes}

In this section we construct a Poincar\'{e} return map for the characteristic foliation on $\partial \mathcal{N}_K$ for strongly null convex subsets $K\subset \real{1,2}$.
The conjugacy class of this map can be used by applying Proposition \ref{prop0} to distinguish different conformal classes on $K$.

For this section assume that $K$ satisfies the conditions of Theorem \ref{thm3}, i.e. the timelike part of $\partial K$ is diffeomorphic to $\real{}\times S^1$ and the lightlike part of $\partial K$ consists of two embedded circles that are the boundary of the timelike part in $\partial K$.

\begin{lemma}\label{lem4}
Let $Y$ be a non vanishing vector field on $\mathbb{T}^2$ and let $(\Gamma_{\theta})_{\theta\in S^1}$ be a foliation by smoothly embedded circles transverse to $Y$.
Then the flow of $Y$ defines a smooth Poincar\'{e} return map on each leaf of the foliation.
The return map of every leaf is in the same (smooth) conjugacy class of circle maps.
\end{lemma}

\begin{proof}
Write $p\in \mathbb{T}^2$ as $p=(\Gamma_{\theta_p}(\phi_p),\theta_p)$, where $\theta_p,\phi_p\in S^1$.
One has to show that for fixed $\theta_0\in S^1 $ and $p_0\in \Gamma_{\theta_0}$ the flow of $Y$ returns to $\Gamma_{\theta_0}$ in finite time.
Define
\begin{align*}
&\abb{f_0}{\real{}}{[0,1]}\\
&t\mapsto \theta_{\Phi_t^Y(p_0)},
\end{align*}
where $\Phi_t^Y$ denotes the flow of $Y$ and we identify $S^1\cong [0,1]/\{0,1\}$.
Since $S^1$ is compact and $Y$ is always transverse to $\Gamma_{\theta}$ one has 
$$f'_0(t)=d\theta(Y(\Phi^Y_t(p_0))>c>0.$$
Hence there exists a $t_0<\infty$ such that $f_p(t_0)=f_p(0)$.
Rescaling $Y$ one can assume that $d\theta(Y(\Phi^Y_t(p_0))=1$.
Then the return map of every leaf is the same.
Scaling the vector field $Y$ by a positive function does not change the conjugacy class.
\end{proof}

Denote the Poincar\'{e} return maps defined in Lemma \ref{lem4} by $\Phi_{Y,\Gamma_{\theta}}$ and their conjugacy class by $[\Phi_{Y,\Gamma}]$.

\begin{prop}
Let $(\Gamma_{\theta})_{\theta\in S^1}$ be a foliation of $\partial\mathcal{N}_K$ by smoothly embedded circles containing the two circles of singularities. 
There exists a non-vanishing vector field $Y\in\Gamma(T\partial \mathcal{N}_K\cap \xi_{\eta})$ which is transverse to all $\Gamma_{\theta}$.
\end{prop}

\begin{proof}
Proposition \ref{lem3} implies that the map $\abb{\sigma}{\partial\mathcal{N}_K}{\mathcal{L}\partial K/\real{}_{>0}}$ is a diffeomorphism.
Denote by $\abb{\pi}{\mathcal{L}\partial K/\real{}_{>0}}{\partial K}$ the projection map.
The circles $\pi(\sigma(\Gamma_{\theta}))$ define a foliation of $T_K\cup L_K$ by embedded circles.
Let $v,w$ be a frame of $\mathcal{L}\partial K/\real{}_{>0}$.
Then $v=v_1+v_2$ and $w=w_1+w_2$, where $v_1(p),v_2(p)\in T_p\partial K $ and $v_2(p), w_2(p)\in T(\mathcal{L}_pN/\real{}_{>0})$.
Assume that $v_1$ is tangent to the foliation $\pi(\sigma(\Gamma_{\theta}))$.
Since $\mathcal{L}_pN/\real{}_{>0}$ is $1$-dimensional one can assume that $w_2=v_2$.
Note that for a vector field $X$ on $\mathcal{L}\partial K$ the pullback $\sigma^{\ast}X$ is Legendrian if and only if $d\pi (X)$ is lightlike.
Moreover $v_1|_{T_K}$ and $w_2|_{T_K}$ are a frame of $T_K$ and $w_2|_{L_K}=0$.

Let $\abb{X_1}{\mathcal{L}\partial K}{\real{}}$ be a smooth function such that $X_1(p)=0$ if and only if $\sigma^{-1}(p)$ is a singularity of the characteristic foliation.
The equations
\begin{align*}
\eta(X_1v_1+X_2w_1,X_1v_1+X_2w_1)&=0\\
X_1^2+X_2^2&=1
\end{align*}
 define a smooth function $\abb{X_2}{\mathcal{L}\partial K}{\real{}}$ .

Then $X:=X_1v+X_2w$ satisfies that $d\pi(X)=0$ and $X$ has no zeroes.
Hence $Y:=\sigma^{\ast}X$ defines a smooth non-vanishing vector field on $\partial\mathcal{N}_K$ contained in $\xi_{\eta}$.
\end{proof}

\begin{remark}
\begin{itemize}
\item[i)]All vector fields with the same property differ from $Y$ by multiplication with a smooth positive function.
The existence of $Y$ implies that the Euler characteristic of $\partial \mathcal{N}_K$ vanishes, i.e. it has to be homeomorphic to the torus.

\item[ii)] The conjugacy class of the Poincar\'{e} return map of the circles of singularities is determined by the lightlike curves on $\partial K$.

Let $[\gamma]\in\partial \mathcal{N}_K$ with $\gamma(0)\in T_K$.
Then $Y_{[\gamma]}$ is an equivalence class of Jacobi fields defined by a variation of null geodesics $\gamma_s$ such that $\gamma_s$ are tangent to $\partial K$ and $\frac{d}{ds}|_{s=0}\gamma_s$ is parallel to $\gamma'(0)$.
Due to Proposition \ref{lem3} the flow of $Y$ projected down to $T_K$ consists of lightlike curves on $\partial K$.
Through each point in $T_K$  run up to re-parametrisation exactly two lightlike curves determined by the two lightlike geodesics tangent to this point.
Furthermore each point of $L_K$ is the endpoint of two different lightlike curves on $\partial K$.
The Poincar\'{e} return map of a singular circle can then be obtained by following a smooth lightlike curve from one circle in $L_K$ to the second circle and going back along the second lightlike direction to the first circle.
\end{itemize}
\end{remark}

\begin{lemma}
Let $\abb{\psi}{\mathcal{N}_{\eta}}{\mathcal{N}_{\eta}}$ be a contactomorphism.
Then $[\Phi_{Y,\Gamma}]=[\Phi_{d\psi(Y),\psi(\Gamma)}]$.
\end{lemma}

\begin{proof}
The contactomorphism $\psi$ maps the characteristic foliation of $\partial\mathcal{N}_K$ to the one of $\psi(\partial\mathcal{N}_K)$ and the circles of singularities to circles of singularities.
Moreover $d\psi(Y)$ is a non-vanishing vector field contained in $\psi^{\ast}\xi_{\eta}$ on $\psi(\partial\mathcal{N}_K)$.
Let $\Gamma_0$ be a circle of singularities.
Let $p\in\Gamma_0$ and $q=\Phi(p)_{Y,\Gamma_0}$.
Then $\psi(q)=\Phi(p)_{d\psi(Y),\psi(\Gamma_0)}$.
Therefore $\Phi(p)_{Y,\Gamma_0}= \psi^{-1}\circ\Phi(p)_{Y,\Gamma_0}\circ \psi$.
\end{proof}

Consider a smooth strictly convex function $\abb{f}{(-1-\epsilon_1,1+\epsilon_2)}{\real{}_{\leq 0}}$ with $f''>0$, $f'(-1)=-1$ and $f'(1)=1$ for some $\epsilon_1,\epsilon_2>0$.
Assume that $\lim\limits_{t\rightarrow -1-\epsilon_1 }f(t)=\lim\limits_{t\rightarrow 1+\epsilon_2 }f(t)=0$.
Moreover assume that  the surface of revolution defined by rotating the graph of $f$ lying in the $(x,t)$-plane in $\real{1,2}$ around the $t$-axis is smooth.
Denote by $K_f$ the open subset of $\real{1,2}$ such that its boundary is this surface of revolution.

\begin{lemma}
The set $K_f$ is strongly null convex.
\end{lemma}

\begin{proof}
Since $f$ is strictly convex, the set $K_f$ is strictly convex in $\real{3}$.
Therefore every null geodesic in $\real{1,2}$ intersects $K_f$ only once, i.e. $\mathcal{N}_{K_f}$ is an open subset of $\mathcal{N}_{\eta}$.
Moreover every null geodesic tangent to $\partial K_f$ intersects $\partial K_f$ in a unique point.

Consider the regular function $H$ given in standard coordinates $(t,x,y)$ by 
$$H(t,x,y):=-f(t)-x^2-y^2.$$
Then for $t\in (-1-\epsilon_1,1+\epsilon_2)$ one has $\partial K_f=H^{-1}(0)$.
The Hessian 
$$\mathrm{Hess}_{\eta}(H)=\left(\begin{array}{ccc}
-f'' & 0 & 0 \\ 
0 & -2 &0 \\ 
0 & 0 & -2
\end{array} \right)$$
is negative definite.
\end{proof}

\begin{remark}
The properties $f'(-1)=-1$ and $f'(1)=1$ implies that $L_{K_f}=\{-1\}\times\{B_{f(-1)}(0)\}\cup \{1\}\times\{B_{f(1)}(0)\}$.
Furthermore $T_{K_f}$ is the surface of revolution defined by $f|_{(-1,1)}$.
In particular the space of null geodesics of $K_f$ only depends on $f|_{(-1,1)}$.
\end{remark}

\begin{prop}
The Poincar\'{e} return map of a singular circle in $\partial\mathcal{N}_{K_f}$ is conjugate to a rotation with angle 
$$2\integ{-1}{1}{\frac{\sqrt{1-(f')^2}}{f}}{t}.$$
\end{prop}

\begin{proof}
Since $T_{K_f}$ is rotational symmetric, the Poincar\'{e} return map of a circle in $L_{K_f}$ does not depend on the choice of starting point and has to be a rotation.
The tangent space for $p\in T_{K_f}$ is spanned by $\frac{\partial}{\partial \theta}$ and 
$$f'(t)\left(\cos (\theta)\frac{\partial}{\partial x}+\sin(\theta)\frac{\partial}{\partial y}\right)+\frac{\partial}{\partial t},$$
where $\theta$ denotes the angle of the polar coordinate in the slices of constant time.
Let $\gamma$ be a lightlike curve on $\partial K_f$ with 
$$\gamma'=\gamma_1'\frac{\partial}{\partial \theta}+\gamma_2'\left(f'(t)\left(\cos (\theta)\frac{\partial}{\partial x}+\sin(\theta)\frac{\partial}{\partial y}\right)+\frac{\partial}{\partial t}\right).$$
Then
$$\eta(\gamma',\gamma')=(\gamma_2'f')^2+(\gamma_1f)^2-(\gamma_2')^2=0.$$
Parametrising $\gamma$ with time one gets
$$(f')^2+(\gamma_1'f)=1.$$
Since $f$ is strictly convex one has $(f'|_{(-1,1)})^2<1$ and $f<0$.
Furthermore $\gamma_1'(t_0)=0$ if and only if $t_0=\pm 1$.
W.l.o.g. assume $\gamma_1'>0$.
Then 
$$\gamma_1'=\frac{\sqrt{1-(f')^2}}{f}.$$
Thus the rotation number of the Poincar\'{e} return map is twice the angle covered by $\gamma$ which is
$$2\integ{-1}{1}{\frac{\sqrt{1-(f')^2}}{f}}{t}.$$
\end{proof}

\begin{exmp}\label{ex2}
Consider $K_f$ for $f(x)=-\frac{1}{4n}x^{2n}+\frac{1}{4}x^2-c$, where $c>\frac{1}{2n}$ is fixed.
Changing $c$, the angle $2\integ{-1}{1}{\frac{\sqrt{1-(f')^2}}{f}}{t}$ can take any value in $(0,2\pi]$.
Hence for every angle one can find an $f$ such that the Poincar\'{e} return map of $K_f$ is conjugate to a rotation by that angle.
\end{exmp}

\begin{exmp}\label{ex3}
Let $C$ be a causal diamond in $\real{1,2}$, i.e. the domain of dependence of the open unit disc in $\{0\}\times \real{2}$.
Then $C$ is globally hyperbolic since $\{0\}\times D^2$ is a Cauchy hypersurface in $C$.
This implies that its space of null geodesics $\mathcal{N}_C$ is contactomorphic to $(ST^{\ast}\real{2},\xi_{st})$.
Its boundary are the fibres of $ST^{\ast}\real{2}$ over the unit circle in $\{0\}\times \real{2}$.
Hence the boundary has a foliation by Legendrian circles.
Write $p\in\partial\mathcal{N}_C$ as $p=(\theta,\psi)$, where $\theta,\psi\in S^1\cong [0,1]/\{0,1\}$.
Then for fixed $\epsilon\in\real{}$ the circle 
$$\{(\theta,\theta+\epsilon)\}\subset \partial\mathcal{N}_C$$
is transverse to the Legendrian circles.
Moreover two of the circles consist of singularities of the characteristic foliation.
The Poincar\'{e} return map of these singular circles is the identity.
\end{exmp}

The examples above provide various strongly causal subsets of $\real{1,2}$ that are diffeomorphic to a ball but not conformally equivalent as Lorentzian manifolds.
On the other hand the conjugacy class of the Poincar\'{e} return map defined above does not uniquely determine the conformal class of the metric.
For suitable choice of the convex function $f$ the Poincar\'{e} return map defined above can be conjugate to the identity map.
In this case $K_f$ is not conformal to the causal diamond since it is not globally hyperbolic.


\end{document}